\documentclass [11pt]{amsart}

\usepackage{graphicx}
\usepackage{amsmath}
\usepackage{amssymb}

\newtheorem{thm}{Corollary}
\newtheorem{Theorem}[thm]{Theorem}

\newtheorem{lemma}{Lemma}
\newtheorem{lemma2}[lemma]{Lemma}
\newtheorem{lemma3}[lemma]{Lemma}
\newtheorem{lemma4}[lemma]{Lemma}
\newtheorem{lemma5}[lemma]{Lemma}
\newtheorem{thm2}[thm]{Theorem}

\newtheorem{corol}[thm]{Corollary}
\newtheorem{corol2}[thm]{Corollary}
\newtheorem{corol3}[thm]{Corollary}
\newtheorem{corol4}[thm]{Corollary}

\newtheorem{rmk1}{Remark}
\newtheorem{rmk2}[rmk1]{Remark}

\newcommand{\R}{{\mathfrak C}}
\newcommand{\A}{\mathcal {A}}
\newcommand{\M}{\mathcal {M}}
\newcommand{\sym}{\textup{sym}}

\begin{document}

\title{Upper bounds on $L$-functions at the edge of the critical strip}

\author{Xiannan Li}
\email{xli@math.stanford.edu}

\address{Department of Mathematics, Stanford University, Stanford, CA 94305}

\date{\today}

\maketitle
\section{Introduction}
In this paper, we are concerned with establishing bounds for $L(1)$ where $L(s)$ is
a general $L$-function, and specifically, we shall be most interested in the case
where no good bound for the size of the coefficients of the $L$-function is known. 
In this case, results are available due to Iwaniec \cite{I1}, \cite{I2}, and Molteni
\cite{Mol}, but this type of investigation is still in its infancy, and the
limitations of the methods available to analytic number theorists is unclear.

The value of $L(1)$ has historically been one of great interest.  The first
interesting examples of bounds for $L(1)$ come from Dirichlet $L$-functions with
non-principal character.  Let $L(s, \chi)$ denote a Dirichlet L-function where
$\chi$ is a non-principal Dirichlet character with modulus $q$.  In the case that
$\chi$ is not a real quadratic character, it is easy to show that 
\begin{equation}
\frac{1}{\log q} \ll L(1, \chi) \ll \log q.
\end{equation}
When $\chi$ is a real quadratic character, 
\begin{equation}
\frac{1}{q^\epsilon} \ll L(1, \chi) \ll \log q,
\end{equation}
where the lower bound is an ineffective result due to Siegel.  By Dirichlet's Class
Number Formula, these also provide information on the class number of
$\mathbb{Q}\:(\sqrt{\chi(-1)q})$, which was one of the classical motivations for
studying $L(1, \chi)$.

Assuming GRH, it can be proven that
\begin{equation}
\frac{1}{\log \log q} \ll L(1, \chi) \ll \log \log q,
\end{equation}
and this is the best possible result.  Indeed, assuming GRH, it is known that $L(1,
\chi)$ is well approximated by a short Euler product, from which the bounds follow.

Far less is known about more general $L$-functions.  As mentioned before, the main
technical difference lies with bounds for the coefficients of the $L$-functions. 
The Dirichlet $L$-function, $L(s, \chi)$, has coefficients which are of size $1$. 
It is conjectured that a more general $L$-function would also have coefficients at
primes bounded by a constant, and this is referred to as Ramanujan's conjecture. 
Assuming Ramanujan's conjecture, one can once again prove bounds of the similar
strength as in (1) save in some exceptional cases in which case the bounds are the
same as in (2).  For instance, if $L(s)$ is an $L$-function attached to a
holomorphic modular form of weight $k \geq 1$ and level $n$, then 
\begin{equation*}
L(s) \ll \log^2 (nk^2).
\end{equation*}

Assuming the Riemann Hypothesis for $L(s)$ in addition to Ramanujan's conjecture
will produce the same type of bounds as in (3).  

Bounds for general $L$-functions will be in terms of $\R$, the conductor of $L(s)$. 
We will describe the conductor in more detail in Section 2, but comment now that in
the case of Dirichlet $L$-functions, the conductor is of size $q$, and for the case
of $L$-functions attached to holomorphic modular forms, the conductor is of size
$nk^2$.  

Unfortunately, Ramanujan's conjecture is far from known for $L$-functions in
general.  In this paper, we are interested in obtaining unconditional bounds for
$L(1)$.  In this direction, Iwaniec \cite{I1} showed for $\pi$ cusp form for
$GL(2)$, and Molteni \cite{Mol} showed for $\pi$ cusp form for $GL(n)$ that
\begin{equation}
L(1, \pi) \ll_\epsilon \R^\epsilon.
\end{equation}
Iwaniec's method cleverly accomplishes this by using the multiplicative relations
between the coefficients.  Essentially if the coefficients of $L$ were too large up
to $O(\R^\epsilon)$, then many coefficients after $\R^\epsilon$ would also be too
large through the multiplicative relations.  Molteni's work is a generalization of
Iwaniec's method to higher dimensions, and involves more complicated algebra and
combinatorics.

For this method to work, both Iwaniec and Molteni used the existence of the
Rankin-Selberg $L$-function $L(s,\pi \times \tilde{\pi})$.  In Theorem \ref{main},
we will not require this assumption to prove that  
\begin{equation}
L(1) \ll \exp\left(C \frac{\log \R}{\log\log \R}\right).
\end{equation}Thus, Theorem \ref{main} holds for $L$-functions for which the
existence of the Rankin-Selberg $L$-function is not known.  We will also prove
stronger bounds such as 
\begin{equation}
L(1, \pi) \ll \exp(C \sqrt{\log \R}),
\end{equation}using the existence of $L(s, \pi \times \tilde{\pi})$.  The proof of
Theorem \ref{main} is analytic, with methods from multiplicative number theory,
notably that of Soundararajan in \cite{Sound}.  

One interesting phenonmenon here is that assuming GRH for $L(s)$ does not appear to
give us significant improvements in upper bounds for $L(1)$, without assuming
Ramanujan.  Recall that assuming GRH in the case of $L$-functions for which
Ramanujan is known gives significant improvements in bounds for $L(s)$, of which
Equation (3) is an example.  This was possible because knowing GRH for $L(s)$ allows
us to express the $L$-function as a short Euler product.  In our case, assuming GRH,
we may still reduce $L(1)$ to a short Euler product, but since we have mediocre
control over the size of the coefficients of $L(s)$, it remains difficult to find a
good upper bound.  We will return to this later.  Now, we further illustrate our
results through a few specific examples.

\subsection{Application to coefficients of Maass forms}
Let $\Gamma = \Gamma_0(N)$ be the Hecke congruence subgroup of level $N$, and $\chi$
an even Dirichlet character with modulus $N$.  Let $S_0(\Gamma, \chi)$ be the space
of Maass cusp forms for $\Gamma$ of weight zero and character $\chi$, and $f\in
S_0(\Gamma, \chi)$ be a newform with eigenvalue $\lambda$ for the Laplacian.  Then
$f$ has a Fourier expansion of the form
$$f(z) = \sum_{n\neq 0} \rho(n) |n|^{-1/2} W(nz),
$$where $W$ is the Whittaker function.  Following standard notation, if we let
$\lambda = 1/4 + t^2$, then in terms of the K-Bessel function $K_{it}$, 
$$W(z) = (|y|\cosh \pi t)^{1/2}K_{it}(2\pi |y|)\exp (2\pi i x).
$$
If we let $a(n)$ denote the eigenvalue of $f$ for the $n^{th}$ Hecke operator $T_n$,
then $\rho(n) = \rho(1)a(n)$.  Thus, $\rho(1) \neq 0$.  

There are two natural ways to normalize $f$.  The first is so that the Peterson norm
$||f|| = 1$, and the second is so that the first Fourier coefficient $\rho(1) = 1 =
a(1)$.  If we normalize $f$ so that $||f|| = 1$, what can the size of $\rho(1)$ be? 
It turns out that this is intimately related to the residue of $L(s, f\times f)$ at
$s=1$ where $f\times f$ is the Rankin-Selberg convolution of $f$ with itself. 
Specifically 
$$
\textup{Res}_{s=1} L(s, f\times f)  = \frac{2\pi}{3}|\rho(1)|^{-2}.
$$
See \cite{HoffLock} for a nice discussion of this.  
As mentioned before, Iwaniec had showed in \cite{I1} that $ |\rho(1)|^2 \gg_\epsilon
(\lambda N)^{-\epsilon}$, which is equivalent to the bound $\textup{Res}_{s=1} L(s,
f\times f) \ll_\epsilon (\lambda N)^\epsilon$.  One of the results in this paper is

\begin{thm}\label{examplecor}
Let $f$ be a Maass cusp form of level $N$, and eigenvalue $\lambda$.  We have that
\begin{equation}
\textup{Res}_{s=1} L(s, f\times f)   \ll\exp \left( C(\log \lambda N)^{1/4}
(\log\log\lambda N)^{\frac{1}{2}} \right).
\end{equation}
for some positive constant $C$, and it follows that  
\begin{equation}
|\rho(1)| \gg_\epsilon \exp( -C'(\log(\lambda N))^{1/4+\epsilon}),
\end{equation}  
where the positive constant $C' = C/2$.
\end{thm}
In \cite{HoffLock}, Hoffstein and Lockhart prove a good lower bound for
$\textup{Res}_{s=1} L(s, f\times f)$, and hence a good upper bound for $|\rho(1)|$. 
Together with the work of Goldfeld, Hoffstein and Lieman in \cite{GHL}, their
results say that 
$$|\rho(1)|^2 \ll \log (\lambda N + 1)$$
where the implied constant is effective, save in the case where $f$ is a lift from
$GL(1)$, in which case the bound is of the form 
$$|\rho(1)|^2 \ll_\epsilon (\lambda N)^{\epsilon},$$ and the constant is no longer
effective.

\subsection{Lower and upper bounds for Maass forms}
One of the motivations for finding upper bounds for $L$-functions at $1$ is that
they are useful for finding corresponding lower bounds.  In \cite{Mol}, Molteni
sketches an approach for proving lower bounds of the form 
\begin{equation}\label{moltenilower}
L(1) \gg_\epsilon (\lambda N)^{-\epsilon}
\end{equation} 
for $L$ corresponding for a Maass cusp form for the Hecke congruence group of level
$N$, where, as before, $\lambda$ is the eigenvalue of the Laplacian for $f$.  The
proof of such lower bounds use upper bounds of the same strength as proven by
Iwaniec in \cite{I1}, but for higher degree $L$-functions than those considered in
\cite{I1}.  Specifically, the relevant $L$-functions are Rankin-Selberg
$L$-functions of the original $L$-function and its symmetric powers.  Molteni's
paper concentrated on proving such upper bounds for general $L$-functions.

Deriving better lower bounds than Equation (\ref{moltenilower}) is difficult even
with better upper bounds in the case where $L$ might have an exceptional zero. 
However, an improvement on the upper bounds will allow improvements on lower bounds
in the cases where exceptional zeros have been ruled out.  For instance, with $f$ as
in \S 1.1 and if $f$ is not a lift from $GL(1)$, we can show that 
\begin{equation}\label{lowerboundeqn}
L(1) \gg \exp(-C \sqrt{\log \lambda N}).
\end{equation}
Moreover, similar results can be shown for other $f$ which are not self-dual.  This
is the subject of Corollary \ref{lowerbound} and Remark 2.  Note that the bound in
\cite{HoffLock} for the special case of $\textup{Res}_{s=1} L(s, f\times f)$ is
superior.  It does not appear to be easily generalized since the proof depends
crucially on the non-negativity of the coefficients.

As would be expected, we obtain better bounds if we have more information about the
functoriality of symmetric powers of $L(s)$.  Improved upper bounds are covered in
our Corollary \ref{sympower}, and we briefly describe the situation with improved
lower bounds in Remark 2.  As an example, let $L(s)$ be automorphic for $GL(2)$ and
$\R$ denote the conductor.  Then based on the work of Kim \cite{Kim} on the
functoriality of the symmetric fourth power for $GL(2)$, it will follow from
Corollary \ref{sympower} that
\begin{equation}\label{betteruppereqn}
L(1) \ll_\epsilon \exp((\log \R)^{1/8 + \epsilon}).
\end{equation}
$\\$

\subsection{Convexity bounds for Rankin-Selberg $L$-functions}
Let $\pi_1$ and $\pi_2$ be cuspidal automorphic representations of $GL(n)$ and
$GL(m)$.  Let $L(s, \pi_1 \times \pi_2)$ denote the Rankin-Selberg L-function and
let $R(\pi_1, \pi_2)$ denote $L(1, \pi_1 \times \pi_2)$ if $L(s, \pi_1 \times
\pi_2)$ has no pole at $s=1$ and the residue at $s=1$ otherwise.  Let $\R$ denote
the conductor of $L(s, \pi_1 \times \pi_2)$.

The standard convexity bound of $R \ll_\epsilon \R^\epsilon$ is known
unconditionally for $m, n \leq 2$ by the work of Molteni \cite{Mol} mentioned above. Molteni proved this result for all Rankin-Selberg $L$-functions satisfying certain strong assumptions on the size of the parameters which are currently unknown beyond $GL(2)$.  Brumley \cite{BrConvex} extended this to $m, n\leq 4$ unconditionally by replacing the condition of Molteni by the existence of a certain strong isobaric lift which is known up to $GL(4)$.  Our Theorem \ref{main} gives that there exists some constant $C>0$ such that
$$R(\pi_1, \pi_2) \ll \exp\left( C \frac{\log \R}{\log \log \R}\right),
$$for all positive integers $m$ and $n$, thus extending this result to all of these Rankin-Selberg $L$-functions.  Further, we note that such a bound is shown in Theorem \ref{main} for all $L$-functions satisfying the standard conditions for an $L$-function.  

Finally, we mention here that this convexity bound for $L(s, \pi_1\times \pi_2)$ has several interesting and immediate applications.  Brumley describes an extension of the zero density result of Kowalski and Michel \cite{KoMi} to all cusp forms on $GL(n)$ over $\mathbb{Q}$ to $n\leq 4$ based on his work in \cite{BrConvex} which we can now extend to all $n$.  See Corollary 4 in \cite{BrConvex} for more details.  In \cite{BrMulOne}, Brumley proves an effective strong multiplicity result which states that two cusp forms are the same if they agree on all the spherical non-archimedean places with norm bounded by $\R^A$ for some $A>0$.  Brumley did not give a specific value for $A$ because in general he had to rely on preconvex bounds for Rankin-Selberg L-functions.  Applying Brumley's method with our bounds would give an improvement in the value of $A$.

\subsection{Overview of paper}
Now we give a very quick overview of the paper.  In \S \ref{formaldefinitions}, we
will introduce the technical formalities needed to state Theorem \ref{main}.  In \S
\ref{proofofmain}, we will prove Theorem \ref{main}.  In \S \ref{betteruppersec}, we
derive further upper bounds, including the stronger bounds in
(\ref{betteruppereqn}).  Lower bounds are the subject of \S \ref{lowerboundssec}, of
which (\ref{lowerboundeqn}) is an example.  In \S\ref{GRHsec}, we briefly discuss
the immediate consequences of GRH for our approach, in terms of both upper bounds
and lower bounds.

\subsection{Formal definitions}\label{formaldefinitions}

We now describe the set of $L$-functions we consider.  They will essentially be
$L$-functions corresponding to automorphic representations of $GL(n)$.  However, it
will be convenient for us to state results which apply to those $L$-functions which
are not known to be automorphic.  We also wish to clarify the conditions needed for
our results.  For these reasons and to fix notation, we now introduce those
conditions which we want our $L$-functions to satisfy.  

\begin{itemize}
\item {\bf Euler Product and Dirichlet Series}. Let $\A = \{\A_p\}$ be a sequence of
square complex matrices of dimension $d$ indexed by primes, with eigenvalues
$\alpha_j(p)$.  Then our general L-function $L(s, \A)$ will be given by 
$$L(s, \A) = \prod_p \prod_{j=1}^d (1-\alpha_j(p)/p^s)^{-1} = \sum_{n\geq
1}\frac{a_n}{n^s},
$$where we assume that the product and the series are absolutely convergent for
$\Re(s) > 1$.  Note that the bound $|\alpha_j(p)| \leq p$ is implied by the
convergence of the Euler product for $\Re s > 1$.

\item  {\bf Analytic Continuation} There is some $m = m(\A)$ such that $L(s, \A)$
can be continued analytically over all of $\mathbb{C}$\: except possibly for a pole
of order $m$ at $s=1$.

\item {\bf Growth} There exists some $\delta>0$ such that $L(s, \A) \ll \exp \exp
(\epsilon |t|)$, for all $\epsilon > 0$, as $|t| \rightarrow \infty$, uniformly in
$-\delta < \sigma < 1+ \delta$.  \footnote{This condition is rather weak, as all
$L$-functions are expected to be of finite order.}
 
\item {\bf Functional Equation}  There exists another sequence of complex matrics
$\A^*$ satifying the above 3 axioms, and there exists $Q_\A$, $\beta_i \in
\mathbb{C}\;$ with $\Re \beta_i > -1/2 \; \forall 1\leq i\leq d$, and $\omega \in
\mathbb{C}$ with $|\omega| = 1$ such that 
$$\Lambda(s, \A) := Q_\A^{s/2} \prod_{i=1}^d \Gamma(s/2 + \beta_i)L(s, \A),
$$and
$$\Lambda(s, \A^*) := Q_\A^{s/2} \prod_{i=1}^d \Gamma(s/2 + \bar{\beta_i})L(s, \A^*),
$$satisfy the functional equation
$$\Lambda(1-s, \A) = \omega \Lambda(s, \A^*).
$$
\end{itemize}

For the purposes of this paper, we define the conductor of $L$ to be 
$$\R = (1+|Q_{\A}|)\prod_{j=1}^d (3+|\beta_j|).
$$

Where it is clear from the context, we will use $Q$, $L(s)$, and $\Lambda(s)$ to
denote $Q_\A$, $L(s, \A)$, and $\Lambda(s, \A)$ respectively.  The above conditions
are a subset of the conditions of Molteni's in \cite{Mol}, and it would be
sufficient also to use conditions such as in 5.1 of \cite{IK}.  The structure of
these conditions are similar to the conditions on the Selberg class \cite{Selberg2}.
 For some more discussion about the types of conditions $L$-functions are expected
to satisfy, refer to either Chapter 5 of \cite{IK}, or \cite{Selberg2}.

The conductor we have defined reflects roughly the complexity of $L(s, \A)$.  For
Dirichlet $L$-functions $L(s, \chi)$, $\R \asymp q$ where $q$ is the modulus of
$\chi$.  For $L$-functions associated to Maass forms in 1.1, $\R \asymp \lambda M$
where $M$ is the level, and $\lambda$ is the eigenvalue.  For Hecke $L$-functions
associated to holomorphic cusp forms, $\R \asymp kM$ where $k$ is the weight and $M$
is the level.  For L-functions associated to cuspical automorphic representations of
$GL(n)$, $\R$ is similar to the analytic conductor defined in Equation (5.7) of \cite{IK}.

The results in this paper are uniform in the parameters $Q$, and $\beta_i$ for
all $i$.  The constants $C$ and the implied constants in Lemma 1-4 and Theorem
\ref{main} are allowed to depend only on $m$.  If we assume that $m\leq d$, then
these constants are absolute.  That $m\leq d$ is known for every $L$-function of
interest.  The constants C and the implied constant in all the
later results depend only on $d$.  There, the L-functions are either automorphic or are
Rankin-Selberg L-functions.  

As we mentioned in \S 1.1,  Molteni \cite{Mol} proved bounds of the form $L(1)
\ll_\epsilon \R^\epsilon$.  Theorem 1 in \cite{Mol} proved this bound for
$L$-functions satisfying our conditions above, having a Rankin-Selberg convolution,
and satisfying a technical condition on the size of the coefficients of $L(s)$ which
Molteni called Hypothsis (R).  Before stating Hypothesis (R), Molteni noted that
some hypothesis about the size of the coefficients must be made.  We do not make any
assumptions concerning the size of the coefficients.  Note however that we do have
the trivial bound $|\alpha_j(p)| \leq p$ for all $j$ which arises merely from the
absolute convergence of the Euler product for $\Re s > 1$.  

In order to prove a bound of the form $\ll \R^\epsilon$ for Rankin-Selberg
$L$-functions (Theorem 2 and Theorem 3 in \cite{Mol}),\footnote{Note that given some
Rankin-Selberg L-function $L(s, \pi_1 \times \pi_2)$ for $\pi_1$ and $\pi_2$
automorphic, it is not known that a Rankin-Selberg $L$-function of the form $L(s,
\pi_1 \times \pi_2 \times \tilde{\pi_1} \times \tilde{\pi_2})$ exists.  Thus the
proof of Theorem 1 in \cite{Mol} does not apply.}  Molteni needed additional
hypotheses on the size of the coefficients.  These additional conditions are known
for Rankin-Selberg $L$-function of automorphic representations for lower rank
groups, but are unproven in general.  See \cite{Mol} for a more detailed
description.

Our first theorem is an improvement of the $\R^\epsilon$ bound in Theorems 1, 2 and
3 in Molteni \cite{Mol}, without assuming the existence of a Rankin-Selberg
convolution or any extra condtions on the coefficients.

\begin{Theorem} \label{main}
Let $L(s)$ be an L-function satisfying our hypothesis above.  Then for any $3 \geq
\sigma \geq 1$, we have that $$(\sigma - 1)^m L(\sigma) \ll \exp \left(Cd \frac{\log
\R}{\log \log \R} \right),$$
where $C>0$ is a constant which can only depend on $m$, and is absolute if $m\leq
d$.\footnote{This is to be interpreted in the obvious manner when $\sigma = 1$.  The
dependence of the right hand side on $d$ is natural, and would be present even when
assuming Ramanujan and GRH.}
\end{Theorem} 
Of course, if $\sigma \geq 3$, then the size of $L(\sigma)$ is well understood.  In
particular, using the trivial bound of $|\alpha_j(p)| \leq p$ and the Euler product,
it is immediate that $L(\sigma) \leq \zeta(2)^d$.

\section{Proof of Theorem \ref{main}}\label{proofofmain}

\subsection{An bound for $\log L(s)$}
We first record some basic facts which follow from our conditions on $L$-functions. 
The function $s^m(s-1)^m\Lambda$ with $\Lambda$ as in the last section is an entire
function of order 1.  Indeed, by standard arguments involving the Phragmen -
Lindelof principle, we know that $L(s)$ is bounded by polynomial growth in $|t|$ for
$\Re{s}$ bounded.  Stirling's formula for $\Gamma$ then gives the result.  Hence we
must have the Hadamard product
$$s^m(s-1)^m\Lambda(s) = e^{A+Bs}\prod_\rho (1-\frac{s}{\rho})e^{s/\rho},
$$where as usual the product is over the non-trivial zeros of $L$ inside the
critical strip $0\leq \Re(s) \leq 1$.  See \cite{Da} for a description of this in
the classical case, which extends to our case without any change, or see Chapter 5
of \cite{IK} for a treatment of this in a general setting.  Logarithmically
differentiating this product for $\Re s > 1$ gives
$$\sum_{\rho} \frac{1}{s-\rho} + \frac{1}{\rho} =  \frac{\log Q}{2} +
\frac{1}{2}\sum_{i=1}^N \frac{\Gamma '}{\Gamma} ( s/2 + \beta_i) + \frac{L'}{L}(s) +
\frac{m}{s-1} + O(1).
$$From this we derive that for $\Re s > 1$,
\begin{equation}\label{logderivative}
- \Re \frac{L'}{L}(s) = -F(s) + \frac{\log |Q|}{2}  + G(s) + \frac{m}{s-1} + O(1),
\end{equation}where
$$F(s) = \sum_\rho \frac{\sigma  - \beta}{(\sigma - \beta)^2 + (t-\gamma)^2}
$$and
$$
G(s) = \sum_{i=1}^N \frac{1}{2} \frac{\Gamma '}{\Gamma} (s/2 + \beta_i).
$$
As usual, we have written $s = \sigma + it$ and non-trivial zeros of $L(s)$ are
written $\rho = \beta + i\gamma$.  We are now ready to proceed to the proof of
Theorem \ref{main}.

We start with a lemma which provides us with a expression for $\frac{L'}{L}(s)$
inside the critical strip in terms of a sum over primes and a sum over zeros.  This
is a version of Lemma 1 of Soundararajan in \cite{Sound} for our generalized
$L$-functions, and is also very similiar to Lemma 2 in \cite{Selberg2}.  Although
the proof of our Lemma is exactly the same as in \cite{Sound}, we provide a sketch
for the sake of completeness.

\begin{lemma}
For $s$ not a zero or pole of $L(s)$, and any $x>1$,
\begin{eqnarray*}
-\frac{L'}{L}(s) &=& \sum_{n\leq x} \frac{\Lambda_{\A}(n)}{n^s}\frac{\log
(x/n)}{\log x} + \frac{1}{\log x}\left(\frac{L'}{L}(s)\right)' \\
&+& \frac{1}{\log x} \sum_{\rho}\frac{x^{\rho - s}}{(\rho - s)^2} -
\frac{mx^{1-s}}{(1-s)^2\log x} + \frac{1}{\log x}\sum_{i=1}^d {\sum_{k\geq 0}}^\flat
\frac{x^{-s - 2(\beta_i + k)}}{(s+2(\beta_i + k))^2},
\end{eqnarray*}
where the $\flat$ over the sum indicates that we sum over $i$ and $k$ such that
$-2(\beta_i + k)$ is a trivial zero of $L(s)$, and $\Lambda_\A(n)$ is the
coefficient which appears in the Dirichlet series for $-\frac{L'}{L}(s)$ to the
right of $\Re s >1$, namely 
$$\Lambda_\A(n) = \begin{cases} 
\log p \sum_{i=1}^d \alpha_i(p)^k &\textup { for } n = p^k\\
0 &\textup{ otherwise}.
\end{cases}  
$$
\end{lemma}
\begin{proof}
Let $c = 1+1/\log x$.  Then, by Perron's formula,
$$\frac{1}{2\pi i} \int_{c-i\infty}^{c+i \infty} -\frac{L'}{L}(s+w)
\frac{x^w}{w^2}dw = \sum_{n\leq x}\frac{\Lambda_{\A}(n)}{n^s}\frac{\log (x/n)}{\log
x}.
$$Now if we move the contour to the left, we pick up various residues at the poles
of the integrand.  There is a residue from the pole at 0, and then there are
residues at the poles corresponding to the possible pole of $L$ at 1, the nontrivial
zeros of $L$, and the trivial zeros of $L$.  Together, these contribute 
$$-\log x \frac{L'}{L}(s) - \left(\frac{L'}{L}(s)\right)' +
\frac{mx^{1-s}}{(1-s)^2}-\sum_\rho\frac{x^{\rho - s}}{(\rho - s)^2} - \sum_{i=1}^d
{\sum_{k\geq 0}}^\flat \frac{x^{-s - 2(\beta_i + k)}}{(s+2(\beta_i + k))^2}.$$  The
justification of this follows from a slight modification of standard arguments from
the original proof of the prime number theorem (see \cite{Da}).  See the proof of
Lemma 2 in \cite{Selberg1} for some more details.  
\end{proof}

Since we are concerned with bounding $L(s)$ near $s = 1$, we integrate the above
Lemma to get an expression for $\log L(s)$ near $s = 1$.  Using Soundararajan's
technique in \cite{Sound}, we will then neglect the sum over non-trivial zeros by
arguing that they give a negative contribution.  This is the subject of our next
Lemma.

\begin{lemma2}
Let $x > 2$, $\lambda > 0$ be such that $e^{-\lambda} \leq \lambda$, and $\sigma_0 =
1 + \lambda/\log x$.  Then
$$\log |L(\sigma_0)| \leq \Re \sum_{n\leq x} \frac{\Lambda_\A(n)\log
(x/n)}{n^{\sigma_0} \log n \log x} + \frac{1}{\log x}\left(\frac{\log |Q|}{2} +
G(\sigma_0) \right) + O\left( d \right),
$$uniformly in $\lambda$.  It is also uniform in $m$ if $m\leq d$.
\end{lemma2}

\begin{proof} 
We integrate both sides of the equation of Lemma 1 from $s = \sigma_0$ to $\infty$. 
To investigate the contribution of the sum over trivial zeros of $L$, recall that
$\Re \beta_i > -1/2$.  Now, we need to distinguish two cases.    

In the first case, we assume that $-2\Re (\beta_i+k) < 1/2$.  Then
$$\int_{\sigma_0}^\infty \frac{x^{-s - 2(\beta_i + k)}}{(s + 2(\beta_i + k))^2}ds
\ll \frac{1}{(1+k^2) \sqrt{x}\log x}.
$$ 
In the second case, assume that $-2\Re (\beta_i+k) \geq 1/2$. Then 
$$\int_{\sigma_0}^\infty \frac{x^{-s - 2(\beta_i + k)}}{(s + 2(\beta_i + k))^2}ds
\ll \frac{1}{(\sigma_0 - 1)^2\log x} \left(x^{-\sigma_0 - 2(\beta_i + k)}\right)\ll
\log x.
$$ We note for future reference that the second case can only occur a finite number
of times, specifically at most $d$ times.  This is because $k \geq 0$ and $\Re
\beta_i > -1/2$, so the second case can only occur if $k=0$.

Next, 
\begin{eqnarray*}
-\int_{\sigma_0}^{\infty} \frac{mx^{1-s}}{\log x(1-s)^2}ds
&\leq& 0.
\end{eqnarray*}

Further, 
\begin{eqnarray*}
\sum_\rho \left|\int_{\sigma_0}^{\infty} \frac{x^{\rho-s}}{(\rho-s)^2}\right| &\leq&
\sum_\rho\frac{x^{\beta - \sigma_0}}{|\sigma_0 - \rho|^2\log x}\\ 
&\leq& \sum_\rho\frac{x^{1-\sigma_0}}{\log x} \frac{1}{(\sigma_0 - \beta)^2 +
\gamma^2}\\
&\leq& \frac{e^{-\lambda}}{\log x (\sigma_0 - 1)}F(\sigma_0)\\
&=& \frac{e^{-\lambda}}{\lambda}F(\sigma_0).
\end{eqnarray*}
Hence,
\begin{eqnarray*}
\log |L(\sigma_0)| &\leq& \Re \sum_{n\leq x} \frac{\Lambda_\A(n)\log
(x/n)}{n^{\sigma_0} \log n \log x} - \Re \frac{1}{\log x} \frac{L'}{L}(\sigma_0)+ 
\\
&&\frac{e^{-\lambda}}{\lambda}\frac{F(\sigma_0)}{\log x } + O\left( d \right).
\end{eqnarray*}
Now we substitute (\ref{logderivative}) for $\frac{L'}{L}(\sigma_0)$ to get
\begin{eqnarray*}
\log |L(\sigma_0)| &\leq& \Re \sum_{n\leq x} \frac{\Lambda_\A(n)\log
(x/n)}{n^{\sigma_0} \log n \log x} + \frac{1}{\log x} \left(-F(\sigma_0) +
\frac{\log |Q|}{2}  + G(\sigma_0) + \frac{m}{\sigma_0-1}\right) + 
\\
&&\frac{e^{-\lambda}}{\lambda}\frac{F(\sigma_0)}{\log x } + O\left( d \right)\\
&\leq& \Re \sum_{n\leq x} \frac{\Lambda_\A(n)\log (x/n)}{n^{\sigma_0} \log n \log x}
+ \frac{1}{\log x} \left(\frac{\log |Q|}{2}  + G(\sigma_0)\right) + O\left( d
\right),
\end{eqnarray*}
where we have used that $\frac{e^{-\lambda}}{\lambda} \leq 1$ and that $F(\sigma_0)
\geq 0$ to see that the $F(\sigma_0)$ terms give a negative contribution.

\end{proof}
It will be convenient to express the bound in Lemma 2 in terms of $\R$, which is the
subject of our next Lemma.
\begin{lemma3}
With the same conditions as in Lemma 2, 
$$\log |L(\sigma_0)| \leq \Re \sum_{n\leq x} \frac{\Lambda_\A(n)\log
(x/n)}{n^{\sigma_0} \log n \log x} + O\left(\frac{\log \R}{\log x} + d\right),
$$where the implied constant is absolute if $m\leq d$.
\end{lemma3}

\begin{proof}
We have that
\begin{eqnarray*}
G(\sigma_0) 
&=& \sum_{i=1}^d \frac{1}{2} \frac{\Gamma '}{\Gamma} (\sigma_0/2 + \beta_i)\\
&\ll& \sum_{i=1}^d \frac{1}{2} \log (\sigma_0/2 + \beta_i) + \sum_{i=1}^d\frac{1}{
\sigma_0 + 2\beta_i}\\
&\ll& \log \prod_{i=1}^d  (\sigma_0/2 + \beta_i) +d \log x\\
&\ll& \log \prod_{i=1}^d  (1 + |\beta_i|) + d\log x,
\end{eqnarray*}uniformly in $\beta_i$.  Here, we have used that $\sigma_0 + 2\beta_i
> \frac{\lambda}{\log x}$.  

From the above we garner
$$\frac{\log |Q| + G(\sigma_0)}{\log x} \ll \frac{\log \R}{\log x} + d,
$$from which the result follows.
\end{proof}

\subsection{Proof of the theorem.}
Now that we have an upper bound for $\log|L(s)|$ near $s=1$, we can use this to
derive an upper bound for $(s-1)^mL(s)$ near $s=1$.  To gain some uniformity in our
results, we first show the following Lemma.  \footnote{Note that the quantity we are
bounding in Lemma 4 is the ratio of the Archimedean factor evaluated at $\sigma_0$
to the Archimedean factor evaluated at $\sigma$.} 
\begin{lemma4}
With $\sigma_0$ as in Lemma 2 and $1 \leq \sigma \leq \sigma_0$, there exists an
absolute constant $C>0$ such that
$$Q^{\frac{\sigma_0 - \sigma}{2}}\prod_{i=1}^d \frac{\Gamma(\sigma_0/2 +
\beta_i)}{\Gamma(\sigma/2 + \beta_i)} \ll \exp\left(C \frac{\log \R}{\log x} +d
\right).
$$
\end{lemma4}
\begin{proof}
Taking the logarithm of both sides, we see that it suffices to show
\begin{equation}\label{Gammaratioeqn1}
\frac{\sigma_0 - \sigma}{2}\log (|Q|+1) + \sum_{i=1}^d \log
\left|\frac{\Gamma(\sigma_0/2 + \beta_i)}{\Gamma(\sigma/2 + \beta_i)}\right| \ll
\frac{\log \R}{\log x} + d.
\end{equation}
In the above, we have inserted $|Q|+1$ in place of $|Q|$ as a matter of convenience.
 Further, it suffices to show that
\begin{equation} \label{Gammaratioeqn2}
\log \frac{\Gamma(\sigma_0/2 + \beta_i)}{\Gamma(\sigma/2 + \beta_i)}
\ll \frac{\log (1+|\beta_i|)}{\log x} + 1
\end{equation}holds for all $i$.
Indeed, summing (\ref{Gammaratioeqn2}) over $i$ and using that $\sigma_0 - \sigma
\ll 1/\log x$ gives
\begin{eqnarray*}
&&\frac{\sigma_0 - \sigma}{2}\log (|Q|+1) + \sum_{i=1}^d \log
\left|\frac{\Gamma(\sigma_0/2 + \beta_i)}{\Gamma(\sigma/2 + \beta_i)} \right|\\
&\ll& \frac{\log (|Q|+1)}{\log x} + \sum_{i=1}^d\frac{\log (1+|\beta_i|)}{\log x} + d\\
&=& \frac{1}{\log x} \log \left((|Q|+1) \prod_{i=1}^d (1+|\beta_i|)\right) + d\\
&\leq&\frac{\log \R}{\log x} + d.
\end{eqnarray*}
Now we prove (\ref{Gammaratioeqn2}).  Let $s = \sigma/2 +\beta_i$, and note that
$\Re s > 0$.  Let $\delta  = \frac{\sigma_0 - \sigma}{2}$.  The ratio
$\frac{\Gamma(s + \delta)}{\Gamma(s)}$ is bounded by a constant for $\Re s \geq 0$
and $|s| \leq 10$.  Note that at $s = 0$, the ratio is actually $0$.  Henceforth
assume $|s| \geq 10$.  Then apply Stirling's formula to see that
\begin{eqnarray*}
\log \frac{\Gamma(s+\delta)}{\Gamma(s)}
&=& (s+\delta - 1/2)\log (s+\delta) - (s-1/2)\log s + O(1)\\
&=& \delta \log (s+\delta)+ O(1).
\end{eqnarray*}Plugging back $s = \sigma/2 + \beta_i$ and noting $\delta \ll 1/\log
x$ gives the result.
\end{proof}

Now we are ready to prove Theorem \ref{main}.
\begin{proof}
First, we will show that $L(\sigma_0) \ll \exp\left(Cd \frac{\log \R}{\log \log \R}
\right)$ for $\sigma_0 = 1+\frac{\lambda}{\log x}$ where $\lambda$ is the unique
real solution to $e^{-\lambda} = \lambda$.  We first note that
$$|\Lambda_\A(p^k)| \leq \log p \sum_{i=1}^d |\alpha_i(p)|^k \leq dp^k\log p.
$$

By Lemma 3, we have
\begin{eqnarray*}
\log |L(\sigma_0)| 
&\leq& \Re \sum_{n\leq x} \frac{\Lambda_\A(n)\log (x/n)}{n^{\sigma_0} \log n \log x}
+ O\left(\frac{\log \R}{\log x} + d\right)\\
&\leq& \sum_{p^k\leq x} d + O\left(\frac{\log \R}{\log x}+d\right)\\
&\ll& \frac{dx}{\log x} + \frac{\log \R}{\log x} + d\\
&\ll& \frac{d\log \R}{\log \log \R},
\end{eqnarray*}upon setting $x = \log \R$.  
Now fix any $\sigma \geq 1$.  If $\sigma \geq \sigma_0$, we are done since the above
proof works as well.\footnote{Lemma 2 was proven uniformly in $\lambda$.}    Say
that $1\leq \sigma < \sigma_0$.  We first note that 
\begin{equation} \label{simpleincreasing}
|\sigma^m (1-\sigma)^m \Lambda(\sigma)| \leq |\sigma_0^m (1-\sigma_0)^m
\Lambda(\sigma_0)|.
\end{equation}
Indeed, set $\Xi(s) = s^m (1-s)^m \Lambda(s)$ and note that $\frac{\Xi '}{\Xi}(s) =
B + \sum_\rho \left(\frac{1}{s-\rho} + \frac{1}{\rho}\right)$ where $\Re B =
-\sum_\rho \Re \frac{1}{\rho}$.  See (7) and (18) in Chapter 5 of \cite{Da} for a
proof of this.  Thus $\Re \frac{\Xi '}{\Xi}(\sigma) = \sum_\rho
\Re\frac{1}{\sigma-\rho} = \sum_\rho \frac{\sigma - \beta}{(\sigma - \beta)^2 +
\gamma^2} \geq 0$ for any $\sigma \geq 1$.  Hence $\log |\Xi(\sigma)|$ and
$|\Xi(\sigma)|$ are increasing.

Finally, we see that (\ref{simpleincreasing}) gives us 
$$(1-\sigma)^m L (\sigma)\ll \left|Q^{\frac{\sigma_0 - \sigma}{2}}\prod_{i=1}^d
\frac{\Gamma(\sigma_0/2 + \beta_i)}{\Gamma(\sigma/2 + \beta_i)}L(\sigma_0)\right|,
$$and the result follows by Lemma 4.
\end{proof}

\section{Rankin-Selberg and refined upper bounds} \label{betteruppersec}
For the rest of this paper, we will restrict our attention to $L$-functions $L(s,
\A) = L(s, \pi)$ for $\pi$ an automorphic representation of $GL(m)$.  This allows us
to apply the well developed theory of the Rankin-Selberg convolution attached to
such representations.  See the appendix of \cite{RudSar} for a brief survey of the
properties of the Rankin-Selberg convolution, and see \cite{JPS} for a far more
detailed exposition.  Here, we remind the reader of its most basic properties.
$\\$
\paragraph{{\bf Rankin-Selberg:}}Say that we have the automorphic $L$-functions 
$$L(s, \A) =\prod_p \prod_{j=1}^d (1-\alpha_j(p)/p^s)^{-1} $$ and 
$$L(s, \M)=\prod_p \prod_{j=1}^d (1-\mu_j(p)/p^s)^{-1}.$$  Then there exists a
finite set of exceptional primes $\mathcal{P}$ such that outside of $\mathcal{P}$,
the local factors of the Rankin-Selberg L-function $L(s, \A\times \M)$ is of the
form
$$L_p(s, \A\times \M) = \prod_{i, j} (1-\alpha_i(p) \mu_j(p)p^{-s})^{-1}.
$$We have moreover that $L(s, \A\times \M)$ satisfies all the conditions introduced
in \S \ref{formaldefinitions}.  We call the primes $p \in \mathcal{P}$ ramified and
note that the coefficients corresponding to a ramified prime $p$ still satifies the
trivial bound of $\leq p$.  

We denote the conductor of this L-function by $\R_{\A\times \M}$.  Then we have 
that $|\mathcal{P}| \ll \log\R_{\A\times \M}$ and $\log\R_{\A\times \M}\asymp \log
\R_{\A} + \log \R_\M$.

\begin{thm2}\label{impupper}
Let $L(s)$ be automorphic of degree $d$ with a pole of order $m$ at $s=1$.  Then for
any $3 \geq \sigma \geq 1$, we have that 
$$(\sigma - 1)^m L(\sigma)  \ll\exp \left( C(\log \R)^{1/2} \right), 
$$where $C>0$ is a constant which depends only on $d$.
\end{thm2}

As we had pointed out before, the constant $C$ and the implied constant in this
result and all the following results can depend only on $d$, and in particular, do not depend
on $m$ since $m \leq d$.  

It is possible to prove better results in the case where more is known about $L(s)$.
 For instance, it will be clear from the proof of Theorem \ref{impupper} that one
can improve the bound to $\ll \exp \left( C(\log \R)^{1/2^l} (\log \log
\R)^{\frac{2^{l}-2}{2^l}}\right)$ if $L(s, \A^{(l)}\times \bar{\A}^{(l)})$ exists. 
Here $\A^{(l)}$ denotes the convolution of $\A$ with itself $l$ times.  Note that
such an assumption immediately gives better upper bounds on the coefficients of
$L(s)$, namely of the from $|a(n)| \ll n^{1/2l}$.  However, directly using such
bounds on the coefficients to find bounds on $L(1)$ does not result in a significant
improvment over Theorem \ref{main}.

In this general theme, here is a result where $L(s, \A)$ is a $GL(2)$ L-function
with a symmetric power $L(s, \sym^l \A)$ such that the Rankin-Selberg L-function
$L(s, \sym^l \A \times \overline{\sym^l \A})$ exists.  \footnote{The corollaries
that follow are corollaries of Theorem \ref{main}.  Theorem \ref{impupper} also
follows from Theorem \ref{main} but we have made the somewhat arbitrary decision not
to call Theorem \ref{impupper} a corollary to emphasize the result.}

\begin{corol}\label{sympower}
Let $L(s, \A)$ be automorphic for $GL(2)$ over $\mathbb{Q}$\: such that both the
$l^{th}$ symmetric power $L(s, \sym^l \A)$ and the Rankin-Selberg L-function $L(s,
\sym^l \A \times \overline{\sym^l \A})$ exists.  Then for any $3 \geq \sigma \geq
1$,
$$(\sigma - 1)^m L(\sigma)  \ll\exp \left( C(\log \R)^{1/2l}
(\log\log\R)^{\frac{l-1}{l}} \right), 
$$where the constant $C$ and the implied constant are absolute.

\end{corol}

\begin{rmk1}
By the recent work by Kim concerning functoriality of the symmetric fourth power
\cite{Kim}, Corollary \ref{sympower} gives us that 
$$(\sigma - 1)^m L(\sigma)  \ll\exp \left( C(\log \R)^{1/8}
(\log\log\R)^{\frac{3}{4}} \right),$$which implies (\ref{betteruppereqn}) in \S 1.2.
 We stated the Corollary in a more general form for the purpose of illustrating that
more information about functoriality of higher symmetric powers leads directly to
improvments on the upper bound.  

Recall that one may prove Ramanujan's conjecture in this situation assuming
functoriality of all symmetric powers.  Given Ramanujan's conjecture, it is fairly
easy to show that $L(1)$ is bounded by a power of $\log \R$.  Corollary
\ref{sympower} also tells us that $L(1)$ is bounded by a power of $\log \R$ if we
can take $l$ large enough, namely around $\log \log \R$.  Thus it is in some sense
an interpolation between the expected bound from Ramanujan and the bound arising
from Theorem \ref{main}.
\end{rmk1}

The original method of Iwaniec which Molteni generalizes in \cite{Mol} depends
crucially on the bound 
$$\sum_{n\leq x} \frac{|a_n|}{n} \ll_\epsilon (x\R)^\epsilon.
$$This type of bound is of some independent interest, and the same idea is used by
Iwaniec in \cite{I2} to prove upper bounds at $1/2$ and by Hoffstein and Lockhart in
\cite{HoffLock} to prove lower bounds at $1$.  Theorem \ref{main} can be used to
derive an improved version of this bound.

\begin{corol2} \label{shortsum}
Let $L(s)$ be automorphic of degree $d$, then
$$\sum_{n\leq x} \frac{|a_n|}{n} \ll \max (\exp (C \sqrt{\log \R}), \exp
(C\sqrt{\log x})),
$$where $C>0$ is a constant depending only on $d$.
\end{corol2}

A better though messier version of the above is possible and will be evident during
the proof of the result.  Once this bound is proven, we may also show similar upper
bounds for $f^{j}(s)$ near $s=1$ where $f(s) = (s-1)^m L(s)$.  Specifically, we have

\begin{corol3} \label{derivative}
Let $L(s)$ be automorphicof degree $d$ and $f(s) = (s-1)^m L(s)$.  Then for $|s-1|
\ll \frac{1}{\log \R}$,
$$f^{j}(s) \ll \exp(C\sqrt{\log \R}),
$$where $C>0$ is a constant.  Here the constants are allowed to depend on $j$ and $d$.
\end{corol3}
The same improvements which appear in Corollary \ref{sympower} hold for Corollaries
\ref{shortsum} and \ref{derivative} should the same assumptions hold, but we will
not show this.

\subsection{Proof of Theorem \ref{impupper}:}
This proof depends on the observation that 
$$\left|\frac{\Lambda_\A(p^k)}{\log p}\right|^2 \leq \frac{\Lambda_{\A \times
\bar{\A}} (p^k)}{\log p},
$$for all primes $p$ and all $k\geq 1$.  See Proposition 6 in \cite{Mol} for a nice
proof of this fact.  Knowing this, we use Cauchy-Schwarz to get our result.
\begin{proof} We will first prove the bound at $\sigma_0$.  We note that 
$$\left|\sum_{\substack{n=p^k\leq x \\ p \textup{ ramified}}} 
\frac{\Lambda_\A(n)\log (x/n)}{n^{\sigma_0} \log n \log x}\right|
\ll (\log\R)^{1/2},
$$where we have used that $\left|\frac{\Lambda_\A(p^k)}{\log p}\right| \ll p^{k/2}$,
and that the number of ramified primes is $\ll \log \R$.  
As before we have
\begin{eqnarray*}
\log |L(\sigma_0)| 
&\ll& \left|\sum_{n\leq x} \frac{\Lambda_\A(n)\log (x/n)}{n^{\sigma_0} \log n \log
x}\right| + \frac{\log \R}{\log x} + 1\\
&\ll& \sum_{\substack{n=p^k\leq x \\ p \textup{ unramified}}} 
\left|\frac{\Lambda_\A(n)\log (x/n)}{n^{\sigma_0} \log n \log x}\right| + (\log
\R)^{1/2} + \frac{\log \R}{\log x}\\
&\leq& \left( \sum_{n\leq x} \frac{|\Lambda_\A(n)|^2}{n^{\sigma_0} \log^2 n}
\right)^{1/2}\left( \sum_{p^k\leq x} \frac{1}{p^{k\sigma_0}} \right)^{1/2} + (\log
\R)^{1/2}+\frac{\log \R}{\log x}\\
&\leq& \left( \sum_{n\leq x} \frac{\Lambda_{\A\times \bar{\A}}(n)}{n^{\sigma_0} \log
n} \right)^{1/2}\left( \sum_{p^k\leq x} \frac{1}{p^{k\sigma_0}} \right)^{1/2} +
(\log \R)^{1/2}+\frac{\log \R}{\log x}\\
&\ll& (\log |L(\sigma_0, {\A \times \bar{\A}})|\log \log x)^{1/2} + \frac{\log
\R}{\log x}+(\log \R)^{1/2}\\
&\ll& \left( \left(\frac{\log \R}{\log \log \R} - k\log (\sigma_0 -1)\right)\log
\log x\right)^{1/2} + \frac{\log \R}{\log x}+(\log \R)^{1/2},
\end{eqnarray*}where $k \leq d^2$ is the order of the pole of $L(s, \A\times
\bar{\A})$ at $1$, and where we have used Theorem \ref{main} to bound $\log
|L(\sigma_0, {\A \times \bar{\A}})|$.  We set $x = \exp (\sqrt{\log \R})$.  Then
$\log \log x = 1/2 \log \log \R$, $- k\log (\sigma_0 -1) \asymp \log \log \R$, and so
$$L(\sigma_0) \ll \exp (C\sqrt{\log \R}).
$$ The result follows by Lemma 4 in the same way as in the proof of Theorem
\ref{main}.  \footnote{Note here that we are applying Lemma 4 to $L(s, \A)$.}
\end{proof}
\subsection{Upper Bounds and Symmetric Powers}
$\\$
The proof of Corollary \ref{sympower} below is essentially the same as the proof of
Theorem \ref{impupper}, where instead of using Cauchy's inequality, we will use
H\"{o}lder's inequality.
\begin{proof}

Let $\mathcal{B}$ be the $l^{th}$ symmetric power of $\A$.  We have
$$L(s, \A) = \prod_{p} (1-\alpha_1(p)/p^s)^{-1}(1-\alpha_2(p)/p^s)^{-1} =
\sum_{n\geq 1}\frac{a(n)}{n^s},
$$and
$$L(s, \mathcal{B}) = \prod_{p} \prod_{j=1}^l
\left(1-\frac{\alpha_1(p)^j\alpha_2(p)^{l-j}}{p^s} \right)^{-1} = \sum_{n\geq
1}\frac{b(n)}{n^s},
$$where for convenience we have ignored that there are finitely many exceptional
Euler factors correponding to ramified primes.  Let $\mathcal{D} = \mathcal{B}
\times \bar{\mathcal{B}}$, and 
$$L(s, \mathcal{D}) = \sum_{n\geq 1}\frac{d(n)}{n^s},
$$
Assume without loss of generality that $|\alpha_1(p)| \geq 1 \geq |\alpha_2(p)|$ for
all $p$.  Say that $p$ is not an exceptional prime.  Let $M = |\alpha_1(p)| =
1/|\alpha_2(p)|$.  We thus have
$$|a(p^k)| \leq 2|\alpha_1(p)|^k = 2 M^k.
$$  Then
\begin{eqnarray*}
|b(p^k)| &=& \left|\sum_{i=0}^l (\alpha_1(p)^{l-i}\alpha_2(p)^i)^k \right|\\
&\geq& M^{kl} - \sum_{i=1}^l M^{k(l-2i)}  \\
&\geq& M^{kl} (1-1/M^k).
\end{eqnarray*}  
If $M^k \geq 2$, then $|a(p^k)|^l \leq 2^{l+1} |b(p)|^k$, and if $M^k \leq 2$, then
$|a(p^k)|^l \leq 2^{2l}.$
Hence
$$|a(p^k)|^l \leq \max(2^{l+1} |b(p^k)|, 4^l).$$ 
Since $|b(p^k)|^2 \leq d(p^k)$, we have
\begin{equation}\label{coeffboundeqn}
|a(p^k)|^{2l} \ll 4^{l} (d(p^k) +4^l)
\end{equation} where the implied constant is absolute.  With $\sigma_0$ as in Lemma
2, we have by use of Lemma 4 as in the proof of Theorem \ref{impupper} that it
suffices to bound $L(\sigma_0)$.  \footnote{Note that here we are applying Lemma 4
to $L(s, \A)$, an L-function of degree $2$, and thus the use of Lemma 4 does not
introduce any dependence on $l$.}  We have
\begin{eqnarray*}
\log |L(\sigma_0)| 
&\ll& \sum_{p^k\leq x} \left|\frac{a(p^k)\log (x/n)}{k p^{k\sigma_0} \log x}\right|
+ \frac{\log \R}{\log x}\\
&\ll& \Re \sum_{\substack{p^k\leq x \\ p \textup{ unramified}}} \frac{|a(p^k)|\log
(x/n)}{k p^{k\sigma_0} \log x} + (\log \R)^{1/2l} + \frac{\log \R}{\log x},
\end{eqnarray*}
where we have used that the number of exceptional primes is $\ll \log \R$, and that
$|a(p^k)| \ll |p^{k/2l}|$ which follows from (\ref{coeffboundeqn}).  Set $x = \R$. 
Again by (\ref{coeffboundeqn}) and H\"{o}lder's inequality, we have
\begin{eqnarray*}
\Re \sum_{\substack{p^k\leq x \\ p \textup{ unramified}}} \frac{|a(p^k)|\log
(x/n)}{k p^{k\sigma_0} \log x}
&\ll& \left(\sum_{p^k\leq x} \frac{d(p^k)+1}{p^{k\sigma_0}}\right)^{1/2l} (\log \log
x)^{\frac{2l-1}{2l}}\\
&\ll& \left( \log|L(\sigma_0, \mathcal{D})| - \log (\sigma_0 - 1) + \log\log x
\right)^{1/2l} (\log \log \R)^{\frac{2l-1}{2l}}\\
&\ll& \left((l+1)^4\frac{\log \R }{\log \log \R}\right)^{1/2l} (\log \log
\R)^{\frac{2l-1}{2l}}\\
&=& (\log \R)^{1/2l} (\log \log \R)^{\frac{l-1}{l}}
\end{eqnarray*}
where yet again we have used Theorem \ref{main} to bound $\sum_{p^k\leq x}
\frac{d(p^k)}{p^{k\sigma_0}}$.  Note that the degree of $L(s, \mathcal{D})$ is
$(l+1)^2$, and $\log \R_\mathcal{D} \ll (l+1)^2 \log \R_\mathcal{A}$.

\end{proof}

The proof of Corollary \ref{examplecor} is the same so we provide a sketch only. 
\begin{proof}
Again, say that
$$L(s, \A) = \prod_{p} (1-\alpha_1(p)/p^s)^{-1}(1-\alpha_2(p)/p^s)^{-1} =
\sum_{n\geq 1}\frac{a(n)}{n^s}.
$$Then
$$L(s, \sym^2 \A) = \prod_{p} \prod_{j=1}^2
\left(1-\frac{\alpha_1(p)^j\alpha_2(p)^{l-j}}{p^s} \right)^{-1} = \sum_{n\geq
1}\frac{s(n)}{n^s},
$$where again for convenience we have ignored that there are finitely many
exceptional Euler factors correponding to ramified primes.  Let $B = \sym^4\A$ so
that 
$$L(s, B) = \prod_{p} \prod_{j=1}^4
\left(1-\frac{\alpha_1(p)^j\alpha_2(p)^{l-j}}{p^s} \right)^{-1} = \sum_{n\geq
1}\frac{b(n)}{n^s}.
$$
Let $\mathcal{D} = \mathcal{B} \times \bar{\mathcal{B}}$, and 
$$L(s, \mathcal{D}) = \sum_{n\geq 1}\frac{d(n)}{n^s}.
$$By the same method as before, we may show that for $n = p^k$ a prime power,
$$|s(n)|^4 \ll |a(n)|^8 \ll |b(n)|^2 \leq d(n),
$$and applying Theorem \ref{main} as before, we obtain 
$$L(1, \sym^2 \A)\ll\exp \left( C(\log \lambda N)^{1/4} (\log\log\lambda
N)^{\frac{1}{2}} \right).
$$  The result follows for $\textup{res}_{s=1} L(s, \A\times \A)$ since the local
factors of $L(s, \A\times \A)$ agree with the the local factors of $\zeta(s) L(s,
\sym^2 \A)$ save at $\ll \log \lambda N$ exceptional primes, which can be neglected
as before.
\end{proof}

\subsection{Proof of Corollary 5 and Corollary 6}
$\\$
The proof of Corollary \ref{shortsum} is also very similar to that of Theorem
\ref{impupper}.  In this proof, we have ignored the ramified primes; the method to
deal with them is the same as in the proof of Theorem \ref{impupper} and Corollary
\ref{sympower}.
\begin{proof}
We have 
\begin{eqnarray*}
\log \sum_{n\leq x} \frac{|a_n|}{n} 
&\leq& \sum_{n\leq x} \frac{|\Lambda_\A(n)|}{n\log n}\\
&\leq& e \sum_{n\leq x} \frac{|\Lambda_\A(n)|}{n^{1+1/\log x}\log n}\\
&\leq& e \left(\sum_{n\leq x} \frac{|\Lambda_\A(n)|^2}{n^{1+1/\log x}\log^2
n}\right)^{1/2} (\log \log x)^{1/2}\\
&\leq& e \log L(1+1/\log x, \A\times \bar{\A})^{1/2} (\log \log x)^{1/2}\\
&\ll& \left( \frac{\log \R}{\log \log \R} + k\log \log x \right)^{1/2} (\log \log
x)^{1/2},
\end{eqnarray*}by Theorem \ref{main}, where $k\leq d^2$ is the order of the pole of
$L(s, \A \times \bar{\A})$ at $s=1$.  The above is $\ll \sqrt{\log \R}$ if $\R\geq
x$ and $\ll \sqrt{\log x}$ otherwise.
\end{proof}

It is clear from the above proof that the stronger bound
$$\sum_{n\leq x} \frac{|a_n|}{n} \ll (\log x)^k \exp\left(C\sqrt{\frac{\log \R}{\log
\log \R}}\right)$$ holds.   As expected, this is actually $\ll \log^{k+\epsilon} x$
for $x$ much larger than $\R$.  We now proceed to  the proof of Corollary
\ref{derivative}, which we will only sketch since it is the same proof as Theorem
\ref{main} in \cite{Mol}.  

\begin{proof}
By Corollary \ref{shortsum}, 
$$\sum_{n\leq x} \log^j n \frac{|a_n|}{n} \ll \log^j x \max(\exp (C\sqrt{\log \R}),
\exp (C\sqrt{\log x})).
$$Define
$$I_j(x) = \frac{1}{2\pi i} \int_{2-i\infty}^{2+i\infty}
\frac{f^{(j)}(s+1)x^s}{s(s+1)...(s+r)} ds,
$$where $r$ is chosen so that the integral converges.  We have that
$\frac{d^j}{ds^j}\left(\frac{s^m}{n^s}\right) = \frac{P_{j, m}(x, \log x)}{n^s}$ 
for a polynomial $P_{j, m}(x, y)$ which is bounded by $m^{j+1}(j+1)x^my^j$.  Then
proceeding as Molteni does, we expand $f^{(j)}(s+1)$ as a series and use a version
of Perron's formula to get that
$$I_j(x) \ll \sum_{n\leq x} \log^j n \frac{|a_n|}{n} \ll \log^j x \max(\exp
(C\sqrt{\log \R}), \exp (C\sqrt{\log x})).
$$

Moreover, by bounding $f^{(j)}(s)$ by $(\R(1+|t|))^c$ using convexity and moving the
integral to the line $\Re s = -1/2$, we get that
$$I_j(x) = \frac{f^{(j)}(1)}{r!} + O \left( \frac{j!\R^{c}}{2^j\sqrt{x}} \right).
$$From this it follows upon choosing $x = \R^{2c + 2}$ that
$$f^{(j)}(1) \ll \exp (C'\sqrt{\log \R}),
$$and we are done.
\end{proof}

\section{Lower Bounds}\label{lowerboundssec}
Part of the original motivation of Molteni's result is to derive Siegel type lower
bounds for $L(1)$.  He showed for specific $L$-functions that $L(1) \gg \frac{1}{\R
^\epsilon}$ (see \cite{Mol} for a more detailed description).  We do not seek to
improve this bound here in the case where $L(s)$ may have a Siegel zero - recall
that such an improvement is extremely difficult even in the case of quadratic
Dirichlet $L$-functions.  Rather, we will examine the case where $L(s)$ has been
proved to have no exceptional zero.

\begin{corol4}\label{lowerbound}
Let $L(s, \A) = L(s, \pi)$ be cuspidal automorphic for $GL(d)$, and assume that
there exists $c_0>0$ such that any zero $\rho = \beta + i\gamma$ of $L(s)$ inside
the critical strip satifies
$$\beta \leq 1 - \frac{c_0}{\log(1+\R |\gamma|)}.
$$Then
$$L(1, \A) \gg \exp (-C \sqrt{\log \R}),
$$where $C>0$ is a constant depending only on $d$.
\end{corol4}

The result above applies to obtain lower bounds on $L(1)$ where $L$ is not self-dual
and automorphic for $GL(m)$ (see Theorem 5.10 of \cite{IK}).  It also applies to the
case of $L(1, f)$ where $f$ is a self-conjugate cuspidal Hecke eigenform on $GL(2)$
by the work of Hoffstein and Ramakrishnan \cite{hofframa} and to $L$-functions of
any cusp form on $GL(3)$ by the work of Banks \cite{ba}.

We obtain lower bounds on $L(s, \A)$ by again expressing $\frac{L'}{L}(s)$ as a sum
over primes and a sum over zeros.  This time however, we will need to genuinely
bound the size of the contribution of the zeros, hence requiring a zero free region.
 We will then bound the contribution of the prime sum using our upper bound results.
 The following Lemma is the same as Lemma 2 of \cite{HoloSound}.

\begin{lemma5}
Let $L(s, \A) = L(s, \pi)$ be automorphic of degree $d$, and assume that there
exists $c_0>0$ such that any zero $\rho = \beta + i\gamma$ of $L(s)$ inside the
critical strip satifies
$$\beta \leq 1- \frac{c_0}{\log (1+\R |\gamma|)}.
$$Let $x = \R^{4\log \log \R /c_0}$.  Then 
$$-\frac{L'}{L}(\sigma) 
= \sum_{n\leq x} \frac{\Lambda_{\A}(n)}{n^s} \left(1-
\left(\frac{n}{x}\right)^2\right) + O(1),
$$uniformly for $1 \leq \sigma < 2$, where the implied constant depends only on $d$.
\end{lemma5}

\begin{proof}
Now fix $\sigma$ with $1 \leq \sigma < 2$.  

Let $c>1$.  We have that
\begin{equation}\label{lowerboundlemmaeqn}
-\frac{1}{2\pi i} \int_{c-i\infty}^{c+i\infty} \frac{L'}{L}(s+\sigma) \frac{2 
x^s}{s(s+2)}ds 
= \sum_{n\leq x} \frac{\Lambda_{\A}(n)}{n^s} \left(1-
\left(\frac{n}{x}\right)^2\right).
\end{equation}
As in the proof of Lemma 1, we shift our contours of integration to $-A$, where $A$
is chosen such that $3/2 \leq A < 2-1/10$ and $L(s)$ has no trivial zeros near the
line $\Re s = -A$.  In particular, we may demand by the pigeonhole principle that
all trivial zeros are at least $\gg 1/d$ away from $-A$.  We encounter a pole at
$0$, and poles at the non-trivial zeros of $L(s)$, and $\ll d$ trivial zeros of
$L(s)$.  We claim that the contribution of the residues from the trivial zeros is
$O(d)$.  To see this, note that each trivial zero is of the form $-2(\beta_i + k)$
for some $1\leq i\leq d$ and integer $k\geq 0$.  We know that $\Re 2 \beta_i \geq
-1/2 + \frac{1}{d^2+1}$ by a result of Luo, Rudnick and Sarnak \cite{LRS}, so that
$-\Re 2(\beta_i + k) \leq-\Re 2 \beta_i\leq 1/2$.  The residue at such a zero is
$$\frac{x^{- 2(\beta_i + k) - \sigma}}{ (\sigma - 2(\beta_i + k))(\sigma - 2(\beta_i
+ k) + 2)}.
$$    Thus $|\sigma - 2(\beta_i + k)| \geq 1/2$.  Moreover, $|\sigma - 2(\beta_i +
k) + 2| \geq 5/2$.  Thus the term above is $O(1)$, and since the number of trivial
zeros between $-A$ and $c$ is $\ll d$, the contribution of all such residues is
$O(d)$ also.

Thus the left hand side of (\ref{lowerboundlemmaeqn}) is
\begin{equation*}
-\frac{L'}{L}(\sigma) - \frac{1}{2\pi i} \int_{-A-i\infty}^{-A+i\infty}
\frac{L'}{L}(s+\sigma) \frac{2x^s}{s(s+2)}ds + O\left( \sum_\rho \frac{x^{\beta -
\sigma}}{|\rho - \sigma||\rho - \sigma + 2|} + d\right).
\end{equation*}
  The integral appearing immediately above may be bounded by bounding
$-\frac{L'}{L}(s+\sigma)$ for $\Re s = -A$.  Indeed, using the functional
equation, and logarithmically differentiating the $\Gamma$ factors, we derive as
usual that $\frac{L'}{L}(s+\sigma) \ll d^2\log (\R(1+|t|))$. \footnote{The $d^2$
factor here arises from the fact that the trivial zeros are $\gg 1/d$ from $-A$
and $d$ is the number of those zeros in an interval of length $2$.  The argument
is a modification of the standard argument on pg. 109 of \cite{Da} for instance.} 
Thus,
$$\frac{1}{2\pi i} \int_{-A-i\infty}^{-A+i\infty} \frac{L'}{L}(s+\sigma) \frac{2
x^s}{s(s+2)}ds \ll d^2 x^{-A}\log \R
$$
Then (\ref{lowerboundlemmaeqn}) gives us
\begin{equation*}
-\frac{L'}{L}(\sigma) = \sum_{n\leq x} \frac{\Lambda_{\A}(n)}{n^s} \left(1-
\left(\frac{n}{x}\right)^2\right) 
+ O\left( \sum_\rho \frac{x^{\beta - \sigma}}{|\rho - \sigma||\rho - \sigma +
2|}\right)+ O(d+d^2x^{-A}\log \R).
\end{equation*}
Now to bound the contribution of the non-trivial zeros, we split the sum into
intervals where $m \leq |\gamma| < m+1$ for $m\geq 0$.  The number of such zeros is
$O(\log \R(m+1))$, and since they lie outside the zero free region, for each $m >
0$, they contribute
$$\ll x^{-\frac{c_0}{\log (\R (1+m))}}  \frac{\log (\R (m+1)) }{m^2}.
$$The term $m = 0$ is of size $\ll x^{-\frac{c_0}{\log \R }} \log^2 \R$.  Thus, the
entire sum over zeros is 
$$\ll x^{-\frac{c_0}{2\log \R}} \log ^2 \R + 1.
$$We set $x = \R^{4\log \log \R /c_0}$ so that the above is $\ll 1$.  Then we have
\begin{eqnarray*}
-\frac{L'}{L}(\sigma) 
&=& \sum_{n\leq x} \frac{\Lambda_{\A}(n)}{n^s} \left(1-
\left(\frac{n}{x}\right)^2\right) + O(1),
\end{eqnarray*}as desired.
\end{proof}

To prove Corollary \ref{lowerbound}, we integrate the expression in Lemma 5.  

\begin{proof}
Set $\sigma_1 = 1 + \exp (-C_0\frac{\log \R}{\log \log \R})$.  
Integrating the expression in Lemma 5 for $\sigma$ from $\sigma_1$ to $2$ gives
\begin{eqnarray*}
|\log L(\sigma_1)|
&=& \sum_{n\leq x} \frac{\Lambda_{\A}(n)}{n^{\sigma_1}\log n} \left(1-
\left(\frac{n}{x}\right)^2\right)| + O(1)\\
&\leq& \sum_{n\leq x} \left|\frac{\Lambda_{\A}(n)}{n^{\sigma_1} \log n }\right| + O(1)\\
&\leq& \left(\sum_{n\leq x} \frac{\Lambda_{\A\times\bar{\A}}(n)}{n^{\sigma_1} \log n
}\right)^{1/2} (\log \log x)^{1/2} + O(1)\\
&\leq& \left( \left(\log L_{\A \times \bar{\A}}(\sigma_1) + \log
\frac{1}{\sigma_1-1} \right) \log \log x  \right)^{1/2} + O(1),
\end{eqnarray*}exactly as in the proof of Theorem \ref{impupper}, where we have used that $L_{\A \times \bar{\A}}$ has a simple pole at $1$.  Using the definition of $\sigma_1$ and also the bound for $\log L_{\A \times \bar{\A}}(\sigma_1)$ from Theorem \ref{main}, we obtain that
$$|\log L(\sigma_1)| \ll \sqrt{\log \R},
$$which implies the lower bound we want at $\sigma_1$.  We use our bound for
$L'(\sigma)$ to push this bound to $s=1$.  Specifically, by Corollary
\ref{derivative}, there exists constants $C_1$ and $C_2$ such that
\begin{eqnarray*}
|L(1)| &\geq& |L(\sigma_1)| - C_1\exp(C_2\sqrt{\log \R}) (\sigma_1 - 1)\\
&=& |L(\sigma_1)| + O\left( \exp(-C_2\frac{\log \R}{\log \log \R}) \right)\\
&\gg& \exp(-C\sqrt{\log \R}),
\end{eqnarray*}as desired, upon setting $C_0 = 2C_2$.
\end{proof}

\begin{rmk2}
The bound proved above can be improved given more information about $L(s)$.  For
instance, results analogous to Corollary \ref{sympower} can be shown using the same
methods.  To be more precise, we can show for $\A$ automorphic for $GL(2)$ over
$\mathbb{Q}$\: that
\begin{equation}
L(1, \A) \gg \exp(-C (\log \R)^{1/8} (\log \log \R)^{3/4}).
\end{equation}  One may show a similar result for the symmetric square.  The proof
is immediate.  By integrating Lemma 5, it suffices to bound the prime sum
$\sum_{n\leq x} \left| \frac{\Lambda_\A(n)}{n^s \log n}\right|$, which is bounded as
in the proof of Corollary \ref{sympower}.  
\end{rmk2}

A standard method to obtain lower bounds from upper bounds depends on constructing
an auxilary L-function $T(s)$ with positive coefficients such that $L(s, \A)$
divides $T(s)$ to a higher power than the pole of $T(s)$ at $s=1$.  In the cases
that this method is available, a zero free region can also be proved, so the method
we used is more general.  

In certain cases, our method applies when the method using auxilary $L$-functions
runs into difficulties.  For instance, in the case of self-dual $GL(2)$
$L$-functions, one may use the auxilary L-function defined in Equation (5.1) in the work of
Hoffstein and Ramakrishnan \cite{hofframa}.  The difficulty with this approach is
that one of the factors in this auxilary L-function is not known to be automorphic,
and thus we cannot obtain upper bounds of sufficient strength to prove Corollary
\ref{lowerbound}.  Using this strategy, only lower bounds of the form $\gg \exp(
-\frac{\log \R}{\log \log \R})$ can be obtained.  On the other hand, the same work
\cite{hofframa} shows that such $L$-functions have the requisite zero free region,
so our result immediately gives that $L(1) \gg \exp( -C\sqrt{\log \R})$.  Actually,
it is not hard to see that bounds of the form $\gg_\epsilon \exp( -C(\log \R)^{1/8 +
\epsilon})$ hold as per the Remark above.  

\section{Some remarks on GRH}\label{GRHsec}

Here, we briefly discuss the immediate consequences of GRH in our setting.  Consider 
\begin{eqnarray*}
-\frac{L'}{L}(\sigma) 
&=& \sum_{n\leq x} \frac{\Lambda_{\A}(n)}{n^s} \left(1-
\left(\frac{n}{x}\right)^2\right) + O\left( \sum_\rho \frac{x^{\beta -
\sigma}}{|\rho - \sigma||\rho - \sigma + 2|}\right) \\
&+& O(1+x^{-A}\log \R),
\end{eqnarray*}
which appears in the proof of Lemma 5.  When we proved Lemma 5, we bounded the sum
over zeros by using the classical zero free region.  Here, we once again split the
sum into intervals where $m \leq |\gamma| < m+1$ for $m\geq 0$.  The number of such
zeros is $O(\log \R(m+1))$, and since they lie on the $\Re s = 1/2$ line by GRH, for
$m > 0$, they contribute
$$\ll x^{-1/2}  \frac{\log (\R (m+1)) }{m^2}.
$$The term $m = 0$ has size $\ll x^{-1/2} \log \R$.  Thus, the entire sum over zeros
is 
$$\ll x^{-1/2} \log \R + 1.
$$Now we are free to set $x = \log^2\R$ so that the above is $\ll 1$.  Then we have
\begin{eqnarray*}
-\frac{L'}{L}(\sigma) 
&=& \sum_{n\leq x} \frac{\Lambda_{\A}(n)}{n^s} \left(1-
\left(\frac{n}{x}\right)^2\right) + O(1),
\end{eqnarray*}as before, but the sum over primes is now much shorter. An
integration as before gives 
\begin{eqnarray*}
\log L(\sigma_1)
&=& \sum_{n\leq x} \frac{\Lambda_{\A}(n)}{n^{\sigma_1}\log n} \left(1-
\left(\frac{n}{x}\right)^2\right) + O(1)
\end{eqnarray*}

This is very much the classical approach to bounding $L(1)$ assuming GRH, and it is
easy to see that if the coefficients $\frac{\Lambda_{\A}(n)}{\log n} \ll 1$ - as in
cases where Ramanujan holds - then the above would be $\ll \log \log \R$.  However,
without Ramanujan, if we have only that $\frac{\Lambda_{\A}(p)}{\log p} \ll
p^\delta$ for all prime $p$ and for some $\delta > 0$, then we can only prove that
\begin{equation}\label{GRHnorambound}
\exp\left(-Cd \frac{(\log \R)^{2\delta}}{\log \log \R}  \right) \ll L(1) \ll
\exp\left(Cd \frac{(\log \R)^{2\delta}}{\log \log \R}  \right).
\end{equation}While it is much easier to prove such a bound assuming GRH than
without it, we see that these bounds are not superior to our unconditional bounds. 
In fact, the bounds in (\ref{GRHnorambound}) are actually somewhat inferior if we
substitute in the best known values of $\delta$.  Although our analysis above is not
detailed, it nevertheless indicates that assuming GRH does not seem to immediately
lead to improvements on our result.  
$\\$
\paragraph{Acknowledgements:} I am very grateful to Professor Soundararajan for his
guidance throughout the making of this paper.  I also thank the referee for a
careful reading of the paper and helpful editorial remarks.

\end{document}